\documentclass[twoside,twocolumn]{article}

  \usepackage{latexsym}
  \usepackage{amsmath}   
  \usepackage{amsfonts}
  \usepackage{amssymb}
  \usepackage{graphicx}
  \usepackage{amsthm}
  \usepackage{url}
  \usepackage[noadjust]{cite}
  \usepackage[bottom=3.5cm, top=2.5cm, left=2.5cm, right=2.5cm, columnsep=0.8cm]{geometry}
  \renewcommand{\mod}{\mathrm{mod}\,\,}
  \newcommand{\alert}{}

  \newcommand{\field}[1]{\mathbb{#1}}
  \newcommand{\C}{\field{C}}
  \newcommand{\R}{\field{R}}

  \newcommand{\BK}{\field{K}}

  \newcommand{\HQ}{\field{H}}
  \newcommand{\vect}[1]{\ensuremath{\mbox{\textbf{\textit{#1}}}}}
  \newcommand{\svect}[1]{\ensuremath{\mbox{\textbf{\textit{\small #1}}}}}
  \newcommand{\bomega}{\boldsymbol{\omega}}

  \newcommand{\be}{\begin{equation}}
  \newcommand{\ee}{\end{equation}} 
  \newcommand{\bv}[1]{\mathbf{#1}}
  \newcommand{\blue}[1]{{#1}}

  \newcommand{\black}[1]{{#1}}
  \newcommand{\bfr}{\begin{frame}}
  \newcommand{\efr}{\end{frame}}

  \newcommand{\btheta}{\boldsymbol{\theta}}
  \newcommand{\sbtheta}{\boldsymbol{\theta}}
  \newcommand{\sbthetap}{\boldsymbol{\theta^{\prime}}}
  \newcommand{\sbthetaz}{\boldsymbol{\theta_0}}
  
  \newtheorem{theorem}{Theorem}
  \newtheorem{definition}{Definition}
  \newtheorem{lemma}{Lemma}
  \newtheorem{corollary}{Corollary}
  
\title{Clifford (Geometric) Algebra Wavelet Transform}

\author{Eckhard Hitzer, Department of Applied Physics, University of Fukui, 910-8507 Japan}

\begin{document}
 
\twocolumn[{\csname @twocolumnfalse\endcsname

\maketitle  
\thispagestyle{empty}
\pagestyle{empty}

\begin{abstract}
\noindent
While the Clifford (geometric) algebra Fourier Transform (CFT) is global, we introduce here the local Clifford (geometric) algebra (GA) wavelet concept.
We show  how for $n=2,3 (\mod 4)$ continuous $Cl_n$-valued admissible wavelets can be constructed using the similitude group $SIM(n)$. We strictly aim for real geometric interpretation, and replace the imaginary unit $i \in \C$ therefore with a GA blade squaring to $-1$. Consequences due to non-commutativity arise. We express the admissibility condition in terms of a $Cl_{n}$ CFT and then derive a set of important properties such as dilation, translation and rotation covariance, a reproducing kernel, and show how to invert the Clifford wavelet transform. As an explicit example, we introduce Clifford Gabor wavelets. We further invent a generalized Clifford wavelet uncertainty principle. Extensions of CFTs and Clifford wavelets to $Cl_{0,n'}, n' = 1,2 (\mod 4)$ appear straight forward.

\vspace{0.5em}



\textbf{AMS Subj. Class.:}
15A66, 42C40, 94A12.

\vspace*{1.0\baselineskip}

\end{abstract}
}]


\section{Introduction}

The meaning and importance of wavelets is clearly seen in a biographical note on J. P. Morlet: 
\textit{Following in the footsteps of Denis Gabor (father of holography), Morlet was disconcerted by the poor results he [Gabor] obtained; but, being inquisitive and persistent, he asked himself, "Why?" and immediately provided the answer. Gabor paved the time-frequency plane in uniform cells and associated each cell with a wave shape of invariant envelope with a carrier of variable frequency. Morlet kept the constraint resulting from the uncertainty principle applied to time and frequency, but he perceived that it was the wave shape that must be invariant to give uniform resolution in the entire plane. For this he adapted the sampling rate to the} 
\noindent
\fbox{
\parbox[c]{7.5cm}{Permission to make digital or hard copies of all or part of
this work for personal or classroom use is granted without
fee provided that copies are not made or distributed for
profit or commercial advantage and that copies bear this
notice and the full citation on the first page. To copy
otherwise, or republish, to post on servers or to
redistribute to lists, requires prior specific permission
and/or a fee.}}
\textit{
 frequency, thereby creating, in effect, a changing time scale producing a stretching of the wave shape. Today the wavelet transform is also called the "time-scale analysis" approach, which is comparable to the conventional time-frequency analysis. \ldots It has been rediscovered as a very useful tool, particularly in data compression where it can produce significant savings in storage and transmission costs but also in mathematics, data processing, communications,
image analysis, and many other engineering problems.}\cite{PG:JPMorlet}

In order to favorably combine wavelet techniques with Clifford (geometric) algebra, which provides a complete algebra of a vector space and all its subspaces, several efforts have been undertaken. 
They include 
Clifford multi resolution analysis (MRA) \cite{MM:CMRA},
quaternion MRA \cite{LT:QW},
Clifford wavelet networks, 
quaternion wavelet transforms (QWT) applied to image analysis (using the QWT phase concept), image processing and motion estimation \cite{EBC:pubs},
quaternion-valued admissible wavelets, Clifford algebra-valued admissible (continuous) wavelets using complex Fourier transforms for the spectral representation \cite{JL:QaAW}, 
monogenic wavelets over the unit ball \cite{KCF:MonW},
Clifford continuous wavelet transforms (ContWT) in $L_{0,2}$, $L_{0,3}$, 
wavelets on the 3D sphere with Cauchy kernel in Clifford analysis (2009),
diffusion wavelets \cite{SB:Wavelets},
ContWT in Clifford analysis, 
{wavelet frames on the sphere,
benchmarking of 3D Clifford wavelet functions,
metric dependent Clifford analysis,
new multivariable polynomials and associated ContWT: 
Clifford versions of Hermite, Hermitean Clifford-Hermite, bi-axial Clifford-Hermite,
Jacobi, Gegenbauer, Laguerre, and Bessel polynomials \cite{Brackx:etal}.

Fourier transformations have been successfully developed in the framework of real Clifford (geometric) algebra (GA), replacing the imaginary unit $i \in \C$ by a geometric (GA) square root of $-1$ \cite{HA:Groot-1}. These Clifford Fourier transformations (CFT) \cite{HM:CFT_UP_Cl3,HM:ICCA7,HMA:2DCWFT} have already found interesting applications in vector field analysis and pattern matching \cite{ES:CFTVectF}. 
A special case are the socalled quaternion Fourier transforms (QFT) \cite{TAE:QFT,EH:QFT,HMAV:WFTquat}.

We now use the spectral CFT representation in order to develop real Clifford GA wavelets in dimensions $n=2,3 (\mod 4)$. We dimensionally extend \cite{MH:CliffWUP} and elaborate the short summary given in \cite{EH:RCliffWT} by adding proofs and generalizations.

In Section 2 Clifford (geometric) algebra is introduced including multivector signal functions, the Clifford Fourier transform, and the similitude group of dilations, rotations and translations. Section 3 defines Clifford mother and daughter wavelets, spectral representation, discusses admissibility, the Clifford wavelet transformation and its spectral CFT representation. This is followed by a detailed discussion of Clifford wavelet properties, i.e. linearity, covariance w.r.t. dilation, rotation and translation, inner product and norm relations, the inverse Clifford wavelet transform, a reproducing kernel and a Clifford wavelet uncertainty principle. Finally the example of Clifford Gabor wavelets is given.

\section{Clifford (geometric) algebra and multivector signals}

\subsection{Clifford (geometric) algebra}

Clifford (geometric) algebra is based on the geometric product of vectors $\vect{a},\vect{b} \in \R^{p,q}, p+q=n$ 
\begin{equation}
  \vect{a}\vect{b} = \vect{a}\cdot\vect{b} + \vect{a}\wedge\vect{b},
\end{equation}
and the associative algebra $Cl_{p,q}$ thus generated with $\R$ and $\R^{p,q}$ as subspaces of $Cl_{p,q}$. $\vect{a}\cdot\vect{b}$ is the symmetric inner product of vectors and $ \vect{a}\wedge\vect{b}$ is Grassmann's outer product of vectors representing the oriented parallelogram area spanned by $\vect{a},\vect{b}$. 

As an example we take the Clifford geometric algebra $Cl_{3}=Cl_{3,0}$ of three-dimensional (3D) Euclidean space $\R^3=\R^{3,0}$.
$\R^3$ has an orthonormal basis $\{\bv{e}_1, \bv{e}_2, \bv{e}_3\}$. 
$Cl_{3}$ then has an eight-dimensional basis of  
\be
  \label{eq:G3basis}
  \{{1}, 
    \underbrace{\bv{e}_1, \bv{e}_2, \bv{e}_3}_{\text{vectors}},
    {\underbrace{\bv{e}_2\bv{e}_3, \bv{e}_3\bv{e}_1, \bv{e}_1\bv{e}_2}_{\text{area bivectors}}},  
    \underbrace{i=\bv{e}_1\bv{e}_2\bv{e}_3}_{\text{volume trivector}}\}.
\ee
Here $i$ denotes the unit trivector, i.e. the oriented volume of a unit cube, with $i^2=-1$. 
The even grade subalgebra $Cl_{3}^+$ is
isomorphic to Hamilton's quaternions $\HQ$. 
Therefore elements of $Cl_{3}^+$ are also called \textit{rotors} (rotation operators), 
rotating vectors and multivectors of $Cl_{3}$. 

In general $Cl_{p,q}, p+q=n$ is composed of so-called $r$-vector subspaces spanned by the induced bases
\be 
  \label{eq:rvecbasis}
  \{\vect{e}_{k_1} \vect{e}_{k_2} \ldots  \vect{e}_{k_r}
   \mid 1 \leq k_1 < k_2 < \ldots < k_r \leq n \},
\ee
each with dimension $\binom{r}{n}$. The total dimension of the $Cl_{p,q}$ therefore becomes $\sum_{r=0}^n \binom{r}{n} = 2^n$.

General elements called \textit{multivectors}
  $M \in Cl_{p,q}, p+q=n,$ have $k$-vector parts ($0\leq k \leq n$):
  \alert{scalar} part
  $Sc(M) = \langle M \rangle = \langle M \rangle_0 = M_0 \in \R$, 
  \alert{vector} part
  $\langle M \rangle_1 \in \R^{p,q}$, 
  \alert{bi-vector} part
  $\langle M \rangle_2$,  \ldots, 
  and
  \alert{pseudoscalar} part $\langle M \rangle_n\in\bigwedge^n\R^{p,q}$
\begin{equation}\label{eq:MVgrades}
    M  =  \sum_{A=1}^{2^n} M_{A} \vect{e}_{A}
       =  \langle M \rangle + \langle M \rangle_1 + \langle M \rangle_2 + \ldots +\langle M \rangle_n \, .
\end{equation}

The \textit{reverse} of $M \in Cl_{p,q}$ defined as
\begin{equation}\label{eq:MVrev}
  \widetilde{M}=\; \sum_{k=0}^{n}(-1)^{\frac{k(k-1)}{2}}\langle M \rangle_k,
\end{equation}
often replaces \blue{complex conjugation and quaternion conjugation}. Taking the reverse is equivalent to reversing the order of products ob basis vectors in the basis blades of \eqref{eq:rvecbasis}. 
The \alert{scalar product} of two multivectors $M, \widetilde{N} \in Cl_{p,q}$ is defined as
\be
    M \ast \widetilde{N} 
    = \langle M\widetilde{N} \rangle 
    = \langle M\widetilde{N} \rangle_0.
\ee
  For $M, \widetilde{N} \in Cl_{n}=Cl_{n,0}$ we get $M\ast \widetilde{N}=\sum_{A} M_A N_A.$
  The \blue{modulus} $|M|$ of a multivector $M \in Cl_{n}$ is defined as 
  \be
     |M|^2 = {M\ast\widetilde{M}}= {\sum_{A} M_A^2}.
  \ee
  For $n=2(\mod 4)$ and $n=3(\mod 4)$ the \blue{pseudoscalar} is
  $i_n=\vect{e}_1\vect{e}_2\ldots\vect{e}_n$ with 
  (also valid for\footnote{%
  For an extension of the current real Clifford (geometric) algebra wavelet approach to Clifford algebras $Cl_{0,n'}$ the definition of reversion has to be modified to include sign changes of negative definite basis vectors $\vect{e}_k \rightarrow \widetilde{\vect{e}}_k = -\vect{e}_k, 1\leq k \leq n'$.}
  $Cl_{0,n'}, \,n'=1,2(\mod 4)$)
  \begin{equation}
     \alert{i_n^2=-1}.
  \end{equation} 
  A \blue{blade} $B_k=\vect{b}_1\wedge\vect{b}_2\wedge\ldots\wedge\vect{b}_k, \vect{b}_l\in\R^{p,q}, 1\leq l \leq k \leq n=p+q$ describes a $k$-dimensional vector \textit{subspace} 
  \be
    V_B=\{ \vect{x}\in \R^{p,q} | \vect{x}\wedge B =0 \}.
  \ee 
  Its \blue{dual blade} 
  \be
  B^{\ast}= Bi_n^{-1}
  \ee
  describes the \textit{complimentary} $(n-k)$-dimensional vector subspace $V^{\perp}_B$. 
  The pseudoscalar $i_n \in Cl_n$ is \textit{central} for $n=3(\mod 4)$
  \be
  \qquad i_n \, M = M \, i_n , \qquad \forall M \in Cl_{n}.
  \ee
But for even $n$ we get due to non-commutativity \cite{HM:ICCA7}
of the pseudoscalar $i_n \in Cl_n $ 
for all $M \in Cl_n, \lambda \in \R$
  \begin{align} \label{eq:evencom}
     i_n  M &= M_{even}  \,i_n - M_{odd} \,i_n\,,  
     \\
     e^{i_n\lambda} M 
     &= M_{even} \;e^{i_n\lambda} + M_{odd} \;e^{-i_n\lambda}.
  \end{align}

\subsection{Multivector signal functions}

\blue{A multivector valued function} 
  $f: \R^{p,q} \rightarrow Cl_{p,q}, \,\, p+q=n, $ has $2^n$ blade components
  $(f_A: \R^{p,q} \rightarrow \R)$
  \begin{equation}\label{eq:MVfunc}
    f(\mbox{\textbf{\textit{x}}})  =  \sum_{A} f_{A}(\vect{x}) {\vect{e}}_{A}.
  \end{equation}
We define the \textit{inner product} of 
$\R^n \rightarrow Cl_{n}$ functions  $f, g $ by
\begin{align}
  \label{eq:mc2}
  (f,g) 
  &= \int_{\R^n}f(\vect{x})
    \widetilde{g(\vect{x})}\;d^n\vect{x}
  \\
  &= \sum_{A,B}\vect{e}_A \widetilde{\vect{e}_B}
    \int_{\R^n}f_A (\vect{x})
    g_B (\vect{x})\;d^n\vect{x},
\end{align}
and the $L^2(\mathbb{R}^n;Cl_{n})$-\textit{norm} 
\begin{align}\label{eq:0mc2}
   \|f\|^2 
   &= \left\langle ( f,f ) \right\rangle
   = \int_{\mathbb{R}^n} |f(\vect{x})|^2 d^n\vect{x}
   \nonumber \\
   &= \sum_{A} \int_{\R^n} f_A^2(\vect{x})\;d^n\vect{x},
   \\
   L^2(\R^n;Cl_{n})
   &= \{f: \R^n \rightarrow Cl_{n} \mid \|f\| < \infty \}. 
\end{align}

For the \textit{Clifford geometric algebra Fourier transformation} (CFT) \cite{HM:ICCA7}
the \blue{complex unit $i \in \C$} is replaced by  
some \textit{geometric} (square) \textit{root} of $-1$, e.g. pseudoscalars $i_n$, $n=2,3(\mod 4)$. 
\blue{Complex functions $f$} are replaced by multivector functions $f \in L^2(\R^{n};Cl_{n})$. 
\begin{definition}[Clifford geometric algebra Fourier transformation (CFT)]  
\label{df:CFT}
The Clifford GA Fourier transform\footnote{ The CFT can be defined analogously for $Cl_{0,n'}, \,n'=\alert{1},\blue{2}(\mod 4)$.} 
  $\mathcal{F} \{f\}$: $\R^n \rightarrow Cl_n, \,n=\alert{2},{3}(\rm mod\,4)$ 
is given by
\begin{equation}\label{eqmk1}
  \mathcal{F}\{f\}(\bomega)
  = \widehat{f}(\bomega)
  = \int_{\R^n} f(\vect{x})
    \,e^{-i_n \bomega \cdot \svect{x}}\, d^n\vect{x},
\end{equation}
for multivector functions $f$: $\R^n \rightarrow Cl_n $. 
\end{definition}
The CFT \eqref{eqmk1} is \textit{inverted} by 
\begin{align}\label{eq11}
  f(\vect{x})
  &= \mathcal{F}^{-1}[\mathcal{F}\{f\}(\bomega )]
  \nonumber
  \\
  &= \frac{1}{(2\pi)^n} \int_{\R^n}\mathcal{F}\{f\}(\bomega ) \, 
    e^{i_n\bomega \cdot \svect{x}}\, d^n \bomega .
\end{align}

The \textit{similitude group} $\mathcal{G}=SIM(n)$ of \alert{dilations, rotations} and \alert{translations} is a subgroup of  the affine group of $\mathbb{R}^n$ 
\begin{align}\label{eq:mc3}
  \mathcal{G} 
  &= \mathbb{R}^+ \times SO(n)\rtimes \mathbb{R}^n 
  \nonumber
  \\
  &= \{(a,r_{\sbtheta}, \vect{b})|
    a \in \mathbb{R}^+,r_{\sbtheta} \in SO(n), 
    \vect{b} \in \mathbb{R}^n \}.
\end{align}
The \textit{left Haar measure} on $\mathcal{G}$ is given by
\begin{align}
  d\lambda 
  &= d\lambda(a,\btheta,\vect{b}) 
  = d\mu(a,\btheta) d^n\vect{b},
  \\
  d\mu 
  &= d\mu(a,\btheta) 
  = \frac{da d\btheta} {a^{n+1}},
\end{align}
where $d\btheta$
is the Haar measure on $SO(n)$. For example
\begin{equation}
  d\btheta = 
  \left\{
  \begin{array}{ll}
     \frac{d\theta}{2 \pi}, & n=2 \vspace*{1mm} \\ 
     \frac{1}{8\pi^2}\sin\theta_1 d\theta_1
  d\theta_2d\theta_3, & n=3
  \end{array}
  \right. \,\, .
\end{equation}
We define the \textit{inner product} of 
$ f,g : \mathcal{G} \rightarrow Cl_{n}$ by
\begin{equation}
  ( f,g ) =
  \int_{\mathcal{G}} f(a, \btheta, \vect{b})
  \widetilde{g(a, \btheta, \vect{b})}\;
  d\lambda(a,\btheta,\vect{b} ),
\end{equation}
and the $L^2(\mathcal{G};Cl_{n})$-\textit{norm} 
\begin{align}
  \|f\|^2 
  &= \left\langle ( f,f )\right\rangle
  = \int_{\mathcal{G}} |f(a,\btheta, \vect{b})|^2 
    d\lambda, 
  \\
  L^2(\mathcal{G};Cl_{n})
  &= \{f: \mathcal{G} \rightarrow Cl_{n} \mid \|f\| < \infty \}. 
\end{align}
The variations of a multivector signal $f\in L^2(\mathbb{R}^n;Cl_{n})$ in position $\vect{x}\in \mathbb{R}^n$ and frequency $\bomega\in \mathbb{R}^n$ are related by the \textit{CFT uncertainty principle}~\cite{HM:ICCA7,HM:UP2005,HM:UP2006,HM:UP2008}
\begin{align}\label{eq:UPCFT}
  &\|\vect{x} f\|_{L^2(\mathbb{R}^n;Cl_{n})}^2\;
  \|\bomega \hat{f}\,\|_{L^2(\mathbb{R}^n;Cl_{n})}^2
  \nonumber
  \\
  & \hspace*{15mm}
  \geq 
  n \frac{(2\pi)^n}{4} \|f\|^4_{L^2(\mathbb{R}^n;Cl_{n})}.
\end{align}

\section{Clifford GA wavelets}

\subsection{Real admissible continuous Clifford GA wavelets}

We represent the transformation group {$\mathcal{G}=SIM(n)$}
by applying
\alert{translations}, \alert{scaling and rotations} to a 
so-called
\textit{Clifford mother wavelet}
$\psi : \R^n \rightarrow Cl_n$
\begin{equation}
  \psi(\vect{x}) \longmapsto 
  \mbox{$\psi$}_{a,\btheta,\svect{b}}
  (\vect{x})
  = \frac{1}{a^{n/2}} \psi 
  (r_{\btheta}^{-1}(\frac{\vect{x}-\vect{b}}{a})).
\end{equation}
The family of wavelets 
$\psi_{a,\btheta,\svect{b}}$
are so-called \textit{Clifford daughter wavelets}.
\begin{lemma}[Norm identity]
The factor ${a^{{-n}/{2}}}$ in $\psi_{ a,\btheta,\svect{b}}$ ensures (independent of $a, \btheta, \vect{b}$) that
\begin{equation}\label{eq:0mc7}
\|\psi_{ a,\btheta,\svect{b}}\|_{L^2(\mathbb{R}^n;Cl_{n})} 
= \| \psi \|_{L^2(\mathbb{R}^n;Cl_{n})}.
\end{equation}
\end{lemma}

\begin{proof}
\begin{align}
  \label{eq:0mc6}
  &\|\psi_{ a,\btheta,\svect{b}}\|_{L^2(\mathbb{R}^n;Cl_{n})}^2
  \nonumber \\
  &=  \int_{\mathbb{R}^n}
  \sum_A \frac{1}{a^n}\psi_A^2 
  (\underbrace{{r_{\btheta}}^{-1}(\frac{\vect{x}-\vect{b}}{a})}_{\displaystyle = \vect{z}})\,d^n\vect{x}
  \nonumber \\
  &= \frac{1}{a^n} \int_{\mathbb{R}^n} \sum_A
  \psi_A^2(\mbox{\boldmath $z$})a^n \det (r_{\btheta})
  \,d^n\mbox{\boldmath $z$}
  \nonumber\\
  &=  \int_{\mathbb{R}^n}\sum_A \psi_A^2(\mbox{\boldmath $z$})\,
  d^n\mbox{\boldmath $z$} 
  = \| \psi \|_{L^2(\mathbb{R}^n;Cl_{n,0})}.
\end{align}
\end{proof}

The \textit{spectral} CFT representation of Clifford daughter wavelets is 
\begin{equation}
  \label{eq:mc7}
  \mathcal{F}\{\mbox{$\psi$}_{ a,\btheta,\vect{b}}\}
   (\bomega) 
  =
  a^{\frac{n}{2}}\widehat{\psi}
  (a r_{\btheta}^{-1}(\bomega))
  e^{-i_n\svect{b}\cdot \bomega} \,.
\end{equation}
In the proof of \eqref{eq:mc7} the CFT properties of scaling, $\vect{x}$-shift and rotation
are applied. 
A Clifford mother wavelet $\psi \in L^2(\R^n;Cl_{n})$
is \textit{admissible} if
\begin{align}
  \label{eq:admconst}
C_{\psi} 
  &=  \int_{\mathbb{R}^+}\int_{S0(n)}a^n \{\widehat{\psi}
  (a r_{\sbtheta}^{-1}(\bomega))\}^{\sim}
  \widehat{\psi}(a r_{\sbtheta}^{-1}(\bomega))\;d\mu
  \nonumber\\
  &= \int_{\mathbb{R}^n} 
  \frac{\widetilde{\widehat{\psi}}(\bomega) \widehat{\psi}(\bomega)}
  {|\bomega|^n} \; d^n\bomega, 
\end{align}
is an \textit{invertible multivector constant and finite} at  a.e. 
$\bomega \in \R^n$.
We must therefore have $\widehat{\psi}(\bomega=0)=0$ 
\begin{align}
  \hat{\psi}(0) 
  &= \int_{\mathbb{R}^n}\psi(\mbox{\boldmath $x$})
  e^{i_n 0\cdot\mbox{\boldmath $x$}}
  \;d^n\mbox{\boldmath $x$} 
  =\int_{\mathbb{R}^n} \psi(\mbox{\boldmath $x$})\;d^n\mbox{\boldmath $x$} 
  \nonumber \\
  &=\sum_A \int_{\mathbb{R}^n} \psi_A (\mbox{\boldmath $x$}) 
   \;d^n\mbox{\boldmath $x$}  \; \mbox{\boldmath $e$}_A 
  = 0,
\end{align}
and therefore for \textit{all} $2^n$ Clifford mother wavelet \alert{components} 
\begin{equation}
  \int_{\mathbb{R}^n} \psi_A (\vect{x}) \;d^n\vect{x}  
  = 0.
\end{equation}
By construction $C_{\psi} = \widetilde{C_{\psi}}$. Hence for $n=2,3 (\mod 4)$
\begin{align}
  &C_{\psi} 
  = \langle C_{\psi} \rangle_0 + \langle C_{\psi} \rangle_1 
    + \langle C_{\psi} \rangle_4 + \langle C_{\psi} \rangle_5 + \ldots
  \nonumber \\  
  &= \sum_{k=0}^{[n/4]}(\langle C_{\psi} \rangle_{4k} +  \langle C_{\psi} \rangle_{4k+1}),
\end{align}
and
\begin{align}
  \langle C_{\psi} \rangle_0  
  &=
  \int_{\mathbb{R}^n} \langle \{\widehat{\psi}
  ({\mbox{\boldmath $\xi $}})\}^{\sim}
  \widehat{\psi}({\mbox{\boldmath $\xi $}}
  )\rangle_0 \;\frac{1}{|{\mbox{\boldmath $\xi $}}|^n}d{\mbox{\boldmath $\xi $}}^n
  \nonumber\\
  &= 
  \int_{\mathbb{R}^n}
  \frac{|\widehat{\psi}({\mbox{\boldmath $\xi $}})|^2}{|{\mbox{\boldmath $\xi $}}|^n}d{\mbox{\boldmath $\xi $}}^n
  > 0 .
\end{align}
The invertibility of $C_{\psi}$ depends on its grade content, e.g. for $n=2,3$, $C_{\psi}$ is
\alert{invertible}, if and only if 
$\langle C_{\psi} \rangle_1^2 \neq \langle C_{\psi} \rangle_0^2 \;$:
\begin{equation}
  \label{eq:ICl}
  C_{\psi}^{-1} = \frac{\langle C_{\psi} \rangle_0 -\langle C_{\psi} \rangle_1 }
  {{\langle C_{\psi} \rangle_0^2 -\langle C_{\psi} \rangle_1^2 } }.
\end{equation}

\begin{definition}[Clifford GA wavelet transformation (CWT) ]
\label{df:GAWT}
For an admissible GA mother wavelet $\psi \in L^2(\mathbb{R}^n;Cl_{n})$ 
and a multivector signal function
$f \in L^2(\mathbb{R}^n;Cl_{n})$
\begin{align}\label{eqq4}
  &T_{\psi} : L^2(\mathbb{R}^n;Cl_{n}) \rightarrow  L^2(\mathcal{G};Cl_{n}),
  \\  
  &f \mapsto T_{\psi}f(a,\btheta,\vect{b})
    = \int_{\mathbb{R}^n}f(\vect{x})
    \widetilde{\psi_{a,\btheta,\vect{b}}(\vect{x})}
    \;d^n\vect{x}.
\end{align}
\end{definition}
\begin{itemize}
\item
Because of \eqref{eq:evencom}
we need to \textit{restrict} the mother wavelet $\psi$ for $n=2 (\mod 4)$ to even or odd grades: 
Either we have a \textit{spinor wavelet} 
$\psi \in L^2(\mathbb{R}^n;Cl^+_{n})$ with $\varepsilon = 1$,  
or we have an \textit{odd parity wavelet}
$\psi \in L^2(\mathbb{R}^n;Cl^-_{n})$ with $\varepsilon = -1$.
\item  
For $n=3 (\mod 4) $, no grade restrictions exist. We then always have $\varepsilon=1$.
\end{itemize}

\noindent
\textbf{NB:} The admissibility constant $C_{\psi}$ is always \textit{scalar} for $n=2$, 
for the spinor wavelet as well as for the odd parity vector wavelet.

The {\textit{spectral} (CFT) representation}
of the Clifford wavelet transform is
\begin{align}
  \label{eq:cwf}
&T_{\psi}f(a,\btheta,\vect{b})
  \\ \nonumber
&= \frac{1}{(2\pi)^n} 
  \hspace*{-1mm}\int_{\mathbb{R}^n}\hspace*{-0.5mm}
  \widehat{f}(\bomega)\,
  a^{\frac{n}{2}} 
  \{\widehat{\psi}(a r_{\btheta}^{-1}(\bomega))\}^{\sim}
  e^{\varepsilon i_n \svect{b}\cdot \bomega} \, d^n\bomega.
\end{align}

\begin{proof} 
\begin{align}
  &{T_{\psi}f(a,\mbox{\boldmath $\theta $},\mbox{\boldmath $b$})}
  \stackrel{IP}= 
  ( f, 
  \mbox{$\psi$}_{ a,\btheta,\svect{b}}
  )_{L^2(\mathbb{R}^n;Cl_{n,0})}
  \\
  &\stackrel{PT} = \frac{1}{(2\pi)^n} 
  ( \widehat{f},\widehat{\mbox{$\psi$}_{a,\btheta,\svect{b}}} 
  )_{L^2(\mathbb{R}^n;Cl_{n,0})}
  \nonumber\\
  &= \frac{1}{(2\pi)^n} \int_{\mathbb{R}^n}\hat{f}
(\mbox{\boldmath $\omega$})
\left[\widehat{\mbox{$\psi$}_{a,\btheta,\svect{b}}}
(\mbox{\boldmath $\omega$})\right]^{\sim}\,
d^n\mbox{\boldmath $\omega$}
  \nonumber\\
  & \stackrel{FT}{=} \frac{1}{(2\pi)^n} 
\int_{\mathbb{R}^n}\hat{f}(\mbox{\boldmath $\omega$})\,
e^{i_n \svect{b}\cdot \mbox{\boldmath $\omega$}} a^{\frac{n}{2}} 
\left[\widehat{\psi}(a r_{\sbtheta}^{-1}
(\mbox{\boldmath $\omega$}))\right]^{\sim}\,
d^n\mbox{\boldmath $\omega$}
  \nonumber\\
  &= \frac{1}{(2\pi)^n} 
\int_{\mathbb{R}^n}\hat{f}(\mbox{\boldmath $\omega$})\,
a^{\frac{n}{2}} 
\left[\widehat{\psi}(a r_{\sbtheta}^{-1}
(\mbox{\boldmath $\omega$}))\right]^{\sim}\,
e^{\varepsilon i_n \svect{b}\cdot \mbox{\boldmath $\omega$}} 
d^n\mbox{\boldmath $\omega$},
\nonumber
\end{align}
with $IP$ = inner product \eqref{eq:mc2}, $PT$ = CFT Plancherel theorem \cite{HM:ICCA7}, and 
$FT=\mathcal{F}\{\mbox{$\psi$}_{a,\btheta,\svect{b}}\}$ of \eqref{eq:mc7}.
\end{proof}
\textbf{NB}: The CFT for $n=2(\mod 4)$ \textit{preserves} even and odd grades.

\subsection{Properties of real Clifford GA wavelets}

We immediately see from Definition \ref{df:GAWT} that the Clifford GA wavelet transform 
is \textit{left linear} with respect to multivector constants $ \lambda_1,\lambda_2 \in Cl_{n}$.

We further have the following set of properties. \textit{Translation covariance}: 
If the argument of $T_{\psi}f(\vect{x})$
is \alert{translated} by a constant $\vect{x}_0 \in \mathbb{R}^n$ then
\begin{equation} 
  [T_{\psi}f(\cdot -\vect{x}_0)](a,\btheta, \vect{b}) 
   = T_{\psi}f(a, \btheta,\vect{b}-\vect{x}_0)\,.
\end{equation}

\begin{proof}
By definition
\begin{align}
&[T_{\psi}f(\cdot -\mbox{\boldmath $x$}_0)]
(a,\mbox{\boldmath $\theta $},\mbox{\boldmath $b$}) 
\nonumber \\
&= \int_{\mathbb{R}^n} f(\mbox{\boldmath $x$} - \mbox{\boldmath $x$}_0) 
\widetilde{\psi_{a,\btheta,\svect{b}}(\mbox{\boldmath $x$})}\,
d^n\mbox{\boldmath $x$}
\nonumber \\
&= \int_{\mathbb{R}^n} f(\underbrace{\mbox{\boldmath $x$}-\mbox{\boldmath $x$}_0}_{ = \vect{y}})
\frac{1}{a^{\frac{n}{2}}}\,
\left[\psi(r_{\sbtheta}^{-1}
(\frac{\mbox{\boldmath $x$}-\mbox{\boldmath $b$}}{a})\right]^{\sim}\,
d^n\mbox{\boldmath $x$}
\nonumber\\
&= \int_{\mathbb{R}^n} f(\mbox{\boldmath $y$})
\frac{1}{a^{\frac{n}{2}}}\,
\left[\psi \left(r_{\sbtheta}^{-1}
(\frac{\mbox{\boldmath $y$}-(\mbox{\boldmath $b$}-
\mbox{\boldmath $x$}_0)}{a})\right)\right]^{\sim}\,
d^n\mbox{\boldmath $y$}
\nonumber\\
&= T_{\psi}f(a, \mbox{\boldmath $\theta$}, \mbox{\boldmath $b$}-
\mbox{\boldmath $x$}_0). 
\end{align} 
\end{proof}

\textit{Dilation covariance}:
If $0<c\in \R$ then
\begin{equation} 
  [T_{\psi}f(c\,\cdot)](a,\btheta,\vect{b}) 
  = \frac {1}{c^{\frac{n}{2}}} T_{\psi}f(ca, \btheta, c\vect{b})\;.
\end{equation}

\begin{proof}
By definition
\begin{align}
&[T_{\psi}f(c\,\cdot)](a,\mbox{\boldmath $\theta $},\mbox{\boldmath $b$}) 
\nonumber\\
&= \int_{\mathbb{R}^n} f(c\mbox{\boldmath $x$})\frac{1}{a^{\frac{n}{2}}}\,
\left[\psi(r_{\sbtheta}^{-1}
(\frac{\mbox{\boldmath $x$}-b}{a}))\right]^{\sim}\,
d^n\mbox{\boldmath $x$}
\nonumber\\
& \stackrel{\mbox{$\vect{y} = c \vect{x}$}}{=} \int_{\mathbb{R}^n} f(\mbox{\boldmath $y$})\frac{1}{a^{\frac{n}{2}}}\frac{1}{c^n}\,
\left[\psi\left(r_{\sbtheta}^{-1}(
\frac{\frac{\svect{y}}{c}-\vect{b}}{a})\right)\right]^{\sim}
\,d^n\mbox{\boldmath $y$}
\nonumber\\
& = \frac {1}{c^{\frac{n}{2}}} \int_{\mathbb{R}^n} f(\mbox{\boldmath $y$})
\frac{1}{(ac)^{\frac{n}{2}}}\,
\left[\psi \left(r_{\sbtheta}^{-1}(
\frac{\mbox{\boldmath $y$}-
\mbox{\boldmath $b$}c}{ac})\right)\right]^{\sim}\,d^n\mbox{\boldmath $y$}
\nonumber\\
& = \frac {1}{c^{\frac{n}{2}}}
T_{\psi}f(ac, \mbox{\boldmath $\theta$},\mbox{\boldmath $b$}c). 
\end{align} 
\end{proof}

\textit{Rotation covariance}:
If  $ r   = r_{\sbtheta} $, 
    $ r_0 = r_{\sbthetaz} $
and $ r'  = r_{\sbthetap} = r_0 r = r_{\sbthetaz} r_{\sbtheta} $
are \alert{rotations}, then
\begin{equation}\label{eq:rotC} 
  [T_{\psi}f(r_{\sbthetaz}\cdot)] (a,\btheta,\vect{b}) 
  = T_{\psi}f(a, \btheta^{\prime}, r_{\sbthetaz}\vect{b})\,.
\end{equation}

\begin{proof}
By definition and with substitution 
$\vect{y} = r_{0} \vect{x}$
\begin{align}
&[T_{\psi}f(r_{0}
\cdot)](a,\btheta,\vect{b}) 
\nonumber \\
&=
\int_{\mathbb{R}^n} f(r_{0}\vect{x}) 
\widetilde{\psi_{a,\btheta,\svect{b}}
(\vect{x})}\,
d^n\vect{x}
\nonumber \\
&=
\int_{\mathbb{R}^n} 
f({r_{0} \vect{x}})
\left[\psi(r_{}^{-1}
(\frac{\vect{x}-\vect{b}}{a}))\right]^{\sim}\,
d^n\vect{x}
\nonumber\\
&=
\int_{\mathbb{R}^n} f(\vect{y})
\left[\psi\left(r_{}^{-1}
(\frac{r_{0}^{-1}
\vect{y} - \vect{b}}
{a})\right)\right]^{\sim}
{\det}^{-1}(r_{})\,d^n\vect{y}
\nonumber\\
&= 
\int_{\mathbb{R}^n} f(\vect{y})
\left[\psi \left(r_{}^{-1}
r_{0}^{-1}
(\frac{\vect{y}-r_{0}\vect{b}}{a})
\right)\right]^{\sim}\,d^n\vect{y}
\nonumber\\
&= 
\int_{\mathbb{R}^n} f(\vect{y})
\left[\psi \left((r_{0}r_{})^{-1}
(\frac{\vect{y}-r_{0}\vect{b}}{a})
\right)\right]^{\sim}\,d^n\vect{y}
\nonumber\\
&=
T_{\psi}f(a, \btheta^{\prime},r_{0}\vect{b}).
\end{align}
\end{proof}

Now we will see some \textit{differences} from the classical wavelet transforms. 

The next property is an 
\textit{inner product relation}:
Let $f, g \in L^2(\mathbb{R}^n;Cl_{n})$ arbitrary. 
Then we have
\begin{eqnarray}
  \label{eqM:C1}
  ( T_{\psi}f,T_{\psi}g )_{L^2(\mathcal{G};Cl_{n})} 
  & = &
  ( fC_{\psi}, g )_{L^2(\mathbb{R}^n;Cl_{n})}
\end{eqnarray}

In the following proof of \eqref{eqM:C1} we will use 
the abbreviations
\begin{align}
  \label{eq:abbrF}
F_{a,{\sbtheta}}(\mbox{\boldmath $\omega$})  
&= 
a^{\frac{n}{2}}\hat{f}(\mbox{\boldmath $\omega$}) 
\{\hat{\psi}(a r_{\sbtheta}^{-1}
(\mbox{\boldmath $\omega$}))\}^{\sim}, 
\\
  \label{eq:abbrG}
G_{a,{\sbtheta}}(\mbox{\boldmath $\omega$}^{\prime})  
&= 
a^{\frac{n}{2}}\hat{g}(\mbox{\boldmath $\omega$}^{\prime}) 
\{\hat{\psi}
(a r_{\sbtheta}^{-1}(\mbox{\boldmath $\omega$}^{\prime}))\}^{\sim}\,,
\end{align} 
and the spectral representations \eqref{eq:cwf}
\begin{align}
  \label{eq:specTf}
T_{\psi}f(a,\mbox{\boldmath $\theta $},\mbox{\boldmath $b$})
  &= \frac{\mathcal{F}\{F_{a,\sbtheta}\}(-\varepsilon \vect{b})}{(2\pi)^n},
\\
  \label{eq:specTg}
T_{\psi}g(a,\mbox{\boldmath $\theta $},\mbox{\boldmath $b$})
  &= \frac{\mathcal{F}\{G_{a,\sbtheta}\}(-\varepsilon \vect{b})}{(2\pi)^n}.
\end{align}

\begin{proof}
Using the abbreviations \eqref{eq:abbrF}, \eqref{eq:abbrG} 
and spectral representations \eqref{eq:specTf}, \eqref{eq:specTg} we get
\begin{align}
&( T_{\psi}f,T_{\psi}g )_{L^2(\mathcal{G};Cl_{3,0})}
\nonumber\\
&=
\frac{1}{(2\pi)^{2n}} 
\int_{\mathbb{R}^+} \int_{S0(n)}
\biggl(\int_{\mathbb{R}^n}\mathcal{F}\{F_{a,\sbtheta}\}(-\varepsilon\mbox{\boldmath $b$})
\nonumber\\
& \hspace*{35mm}
\left\{\mathcal{F}\{G_{a,\sbtheta}\}(-\varepsilon\mbox{\boldmath $b$})\right\}^{\sim}
d^n\mbox{\boldmath $b$} \biggr)d\mu
\nonumber\\
&\stackrel{PT}{=}
\int_{\mathbb{R}^+} 
\int_{S0(n)} \frac{1}{(2\pi)^{n}}
\left(\int_{\mathbb{R}^n}F_{{a,\sbtheta}}(\mbox{\boldmath $\xi$})
\widetilde{G_{a,\sbtheta} (\mbox{\boldmath $\xi$})}\,
d^n\mbox{\boldmath $\xi$} \right)
\,d\mu
\nonumber \\
&=
\frac{1}{(2\pi)^{n}}
\int_{\mathbb{R}^n} 
\bigg( \int_{\mathbb{R}^+} a^n \int_{S0(n)} \hat{f}(\mbox{\boldmath $\xi$})
\{\hat{\psi}(ar_{\sbtheta}^{-1}
(\mbox{\boldmath $\xi$}))\}^{\sim}\,
\nonumber \\
& \hspace*{40mm}
\hat{\psi}(ar_{\sbtheta}^{-1}(\mbox{\boldmath $\xi$}))
\widetilde{\hat{g}(\mbox{\boldmath $\xi$})}
d^n\mbox{\boldmath $\xi$}
\bigg)
\,d\mu
\nonumber\\
&=
\frac{1}{(2\pi)^{n}}
\int_{\mathbb{R}^n}\hat{f}(\mbox{\boldmath $\xi$})
\bigg( \int_{\mathbb{R}^+} \int_{S0(n)} a^n \{\hat{\psi}
(ar_{\sbtheta}^{-1}(\mbox{\boldmath $\xi$}))\}^{\sim}
\nonumber \\
& \hspace*{30mm}
\hat{\psi}(ar_{\sbtheta}^{-1}(\mbox{\boldmath $\xi$}))
d\mu 
\bigg)_{\alert{= \,C_{\psi}}}\hspace*{-2mm}
\widetilde{\hat{g}(\mbox{\boldmath $\xi$})}\,
d^n\mbox{\boldmath $\xi$}
\nonumber\\
&=
\frac{1}{(2\pi)^{n}}
\int_{\mathbb{R}^n} \hat{f}(\mbox{\boldmath $\xi$})C_{\psi} 
\widetilde{\hat{g}(\mbox{\boldmath $\xi$})}\,
d^n\mbox{\boldmath $\xi$}
\nonumber\\
&\stackrel{PT}{=}
\int_{\mathbb{R}^n} f(\mbox{\boldmath $x$})C_{\psi} 
\widetilde{g(\mbox{\boldmath $x$})}\,
d^n\mbox{\boldmath $x$} 
= 
( fC_{\psi} , g )_{L^2(\mathbb{R}^n;Cl_{n,0})},  
\end{align} 
where $PT$ denotes the CFT Plancherel theorem.
For the second equality we have also used the fact, that a substitution $\vect{b}'=-\varepsilon\vect{b}, \varepsilon=\pm 1$, as in 
$\int_{\R^n}h(-\varepsilon\vect{b})\,d^n\vect{b}
= \int_{\R^n}h(\vect{b}')d^n\vect{b}'$, does not change the overall sign. 
\end{proof}

As a corollary we get the following \textit{norm relation}:
\begin{align}\label{eq1:col}
\| T_\psi f \|_{L^2(\mathcal{G};Cl_{n})}^2 
  &=
  Sc ( f C_{\psi} , f )_{L^2(\mathbb{R}^n;Cl_{n})}
  \\
  &= C_{\psi}\ast ( f , f )_{L^2(\mathbb{R}^n;Cl_{n})} \,.  
\end{align}
We can further derive the
\begin{theorem}[Inverse Clifford $ Cl_{n}$ wavelet transform]
Any $ f \in L^2(\mathbb{R}^n;Cl_{n})$ can be decomposed with respect to
an admissible Clifford GA wavelet as
\begin{align}\label{eq:invCWT}
  f(\vect{x})  
  &= \int_{\mathcal{G}}
    T_{\psi}f (a,\btheta,\vect{b})\,
    \psi_{a,\btheta,\svect{b}}\,
    C_{\psi}^{-1}\,d\mu d^n\vect{b}
  \\ \nonumber
  &= \int_{\mathcal{G}}
    ( f, \psi_{a,\btheta,\svect{b}}
    )_{L^2(\mathbb{R}^n;Cl_{n})}
    \psi_{a,\btheta,\svect{b}}
    C_{\psi}^{-1}\,d\mu d^n\vect{b} ,
\end{align}
the integral converging in the weak sense.
\end{theorem}

\begin{proof}
For any $g\in L^2(\mathbb{R}^n;Cl_{n,0})$
\begin{align}
&( T_{\psi}f,T_{\psi}g )_{L^2(\mathcal{G};Cl_{n,0})} 
\nonumber\\
&=
\int_{\mathcal{G}}
T_{\psi}f (a,\mbox{\boldmath $\theta$},\mbox{\boldmath $b$})
\{T_{\psi}g (a,\mbox{\boldmath $\theta$},\mbox{\boldmath $b$})\}^{\sim}\,
d\mu d^n\mbox{\boldmath $b$}
\nonumber\\
&=
\int_{\mathcal{G}} \int_{\mathbb{R}^n} T_{\psi}f 
(a,\mbox{\boldmath $\theta$},\mbox{\boldmath $b$})
\psi_{a,\btheta,\svect{b}}(\mbox{\boldmath $x$})
\widetilde{g(\mbox{\boldmath $x$})}\,
d^n\mbox{\boldmath $x$}d\mu d^n\mbox{\boldmath $b$}
\nonumber\\
&=
\int_{\mathbb{R}^n} \int_{\mathcal{G}}  
T_{\psi}f (a,\mbox{\boldmath $\theta$},\mbox{\boldmath $b$})
\psi_{a,\btheta,\svect{b}}(\mbox{\boldmath $x$})
\,
d\mu d^n\mbox{\boldmath $b$}\,
\widetilde{g(\mbox{\boldmath $x$})}\,d^n\mbox{\boldmath $x$}
\nonumber\\
&=
\left(\alert{\int_{\mathcal{G}}
T_{\psi}f (a,\mbox{\boldmath $\theta$},\mbox{\boldmath $b$})
\psi_{a,\btheta,\svect{b}}\,
d\mu d^n\mbox{\boldmath $b$}}\,,\, g \right)_{L^2(\mathbb{R}^n;Cl_{n,0})}
\nonumber \\
&\stackrel{IPR}{=} 
( \alert{fC_{\psi}}\,,\, g )_{L^2(\mathbb{R}^n;Cl_{n,0})},
\label{eq:invproof}
\end{align}
where $IPR$ denotes the inner product relation \eqref{eqM:C1}.
Because \eqref{eq:invproof} holds for any $g\in L^2(\mathbb{R}^n;Cl_{n,0})$ we get
\begin{align}
  f(\mbox{\boldmath $x$}) C_{\psi} =
  \int_{\mathcal{G}} T_{\psi}f (a,\mbox{\boldmath $b$},\mbox{\boldmath $\theta$})
  \psi_{a,\btheta,\svect{b}}(\mbox{\boldmath $x$})\,
  d\mu d^n\mbox{\boldmath $b$},
\end{align}
or equivalently the \textit{inverse CWT}
\begin{align}
  f(\mbox{\boldmath $x$}) = 
  \int_{\mathcal{G}} T_{\psi}f (a,\mbox{\boldmath $b$},\mbox{\boldmath $\theta$})
  \psi_{a,\btheta,\svect{b}}(\mbox{\boldmath $x$})\,
 C_{\psi}^{-1}\,d\mu d^n\mbox{\boldmath $b$}. 
\end{align}
\end{proof}

Next is the \textit{reproducing kernel}:
We define for an admissible Clifford mother wavelet $\psi \in L^2(\mathbb{R}^n;Cl_{n})$
\begin{align}
  &\BK_{\psi}(a,\btheta,\vect{b};
  a^{\prime},\btheta^{\prime},\vect{b}^{\prime}) 
  \nonumber \\
  &\quad = 
  \left( 
    \psi_{a,\btheta,\svect{b}}
    C_{\psi}^{-1} , 
    \psi_{a^{\prime},\btheta^{\prime},\svect{b}^{\prime}}
  \right)_{L^2(\mathbb{R}^n;Cl_{n})} \,.
\end{align}
Then 
$\BK_{\psi}(a,\btheta,\vect{b};a^{\prime},\btheta^{\prime},
\vect{b}^{\prime})$
is a reproducing kernel in 
$L^2(\mathcal{G}, d\lambda )$,
i.e,
\begin{align}
  &T_{\psi} f (a^{\prime}, \btheta^{\prime},\vect{b}^{\prime}) 
  \nonumber \\
  &=
  \int_{\mathcal{G}}T_{\psi}f (a,\btheta,\vect{b})
  \BK_{\psi}(a,\btheta,\vect{b};a^{\prime},\btheta^{\prime},\vect{b}^{\prime}) 
  d\lambda \,.
\end{align}

\begin{proof}
By inserting for $f(\vect{x})$ the inverse CWT \eqref{eq:invCWT} into the definition
of the CWT we obtain
\begin{align}
& T_{\psi} f (a^{\prime}, 
\mbox{\boldmath $\theta$}^{\prime},\mbox{\boldmath $b$}^{\prime})
\nonumber\\
&=
\int_{\mathbb{R}^n}\left( \int_{\mathcal{G}}
T_{\psi}f (a,\mbox{\boldmath $\theta$},\mbox{\boldmath $b$})\;
\psi_{a,\btheta,\svect{b}}(\mbox{\boldmath $x$})
\,d\lambda\; C_{\psi}^{-1}\right)
\nonumber\\
& \hspace*{48mm}
\{\psi_{a^{\prime},\btheta^{\prime},\svect{b}^{\prime}}
(\mbox{\boldmath $x$})\}^{\sim}
\,d^n\mbox{\boldmath $x$}
\nonumber\\
&= \int_{\mathcal{G}}
T_{\psi}f (a,\mbox{\boldmath $\theta$},\mbox{\boldmath $b$})
\nonumber\\
& \hspace*{3mm}
\left( \int_{\mathbb{R}^n} 
\psi_{a,\btheta,\svect{b}}(\mbox{\boldmath $x$})C_{\psi}^{-1}
\{\psi_{a^{\prime},\btheta^{\prime},\svect{b}^{\prime}}
(\mbox{\boldmath $x$})\}^{\sim}
\,d^n\mbox{\boldmath $x$} \right)_{\alert{= \BK_{\psi}}}
d\lambda
\nonumber\\
&= \int_{\mathcal{G}} 
T_{\psi}f (a,\mbox{\boldmath $b$},\mbox{\boldmath $\theta$})
\BK_{\psi}(a,\mbox{\boldmath $\theta$},\mbox{\boldmath $b$};
a^{\prime},\mbox{\boldmath $\theta$}^{\prime}, \mbox{\boldmath $b$}^{\prime})\,
d\lambda .
\end{align}
\end{proof}

\begin{theorem}[Generalized Clifford GA wavelet uncertainty principle]
\label{th4.1}
Let $\psi $ be an admissible Clifford algebra mother wavelet.
Then for every $f \in L^2(\mathbb{R}^n;Cl_{n})$, 
the following inequality holds 
\begin{align}
  \label{eq:genCWUP}
   &\| \vect{b} T_{\psi}f(a,\btheta,\vect{b}) \|^2_{L^2({\mathcal G};Cl_{n})}
   C_{\psi} \ast 
   (\widetilde{\bomega\hat{f}} , \widetilde{\bomega \hat{f}})_{L^2(\mathbb{R}^n;Cl_{n})}
   \nonumber \\
   &\;\;\;\geq  
   \frac{n(2\pi)^n}{4}\,\,\left[C_{\psi}\ast ( f,f )_{L^2(\mathbb{R}^n;Cl_{n})}\right]^2.
\end{align}
\end{theorem} 
\noindent
\textbf{NB:} The integrated variance
\be
  \int_{\R^+}\int_{SO(n)}
  \|\bomega \mathcal{F}\{T_{\psi}f(a,\btheta, \,.\, ) \} \|^2_{L^2(\mathbb{R}^n;Cl_{n})} d \mu
\ee
is \textit{independent} of the wavelet parity $\varepsilon$. Otherwise the {proof} is similar to the one for $n=3$ in \cite{MH:CliffWUP}. For scalar admissibility constant this reduces to

\begin{corollary}[Uncertainty principle for Clifford GA wavelet]
  \label{cor:scUP}
Let $\psi $ be a Clifford algebra  wavelet with scalar 
admissibility constant $C_{\psi}\in\R^n$.
Then for every $f \in L^2(\mathbb{R}^n;Cl_{n})$, the following inequality holds 
\begin{align}\label{eq:UPCWTs}
  &\|\vect{b} \,T_{\psi} f(a,\btheta,\vect{b})\|_{L^2(\mathcal{G};Cl_{n})}^2\;
  \|\bomega \hat{f}\|_{L^2(\mathbb{R}^n;Cl_{n})}^2
  \nonumber
  \\
  & \hspace*{15mm}
  \geq 
  n C_{\psi} \frac{(2\pi)^n}{4} \|f\|^4_{L^2(\mathbb{R}^n;Cl_{n})}.
\end{align}
\end{corollary}

\begin{itemize}
\item 
This shows indeed, that Theorem \ref{th4.1} 
represents a multivector generalization
of the uncertainty principle 
for Clifford wavelets with scalar admissibility constant.
\item 
Compare with the (direction independent) uncertainty principle \eqref{eq:UPCFT} for the CFT.
\end{itemize}

\subsection{Example of Clifford GA Gabor wavelets}

Finally \textit{Clifford (geometric) algebra Gabor Wavelets} are defined as 
(variances $\sigma_k, 1 \leq k \leq n$, for
$n=2(\mod 4): A \in Cl^+_n$ or $A\in Cl^-_n$)
\begin{align}
  \label{eq:CGW}
  \psi^{c}(\vect{x}) 
  &= 
  \frac{A \,\, e^{-\frac{1}{2}\sum_k\frac{x_k^2}{\sigma_k^2} }}{(2\pi)^{\frac{n}{2}} \prod_{k=1}^n \sigma_k }\,
(
  e^{i_n \bomega_0\cdot \svect{x}}
  - \underbrace{e^{-\frac{1}{2}\sum_{k=1}^n \sigma_k^2 \omega_{0,k}^2}}_{\mbox{constant}}
),
\nonumber
\\
&\vect{x},\bomega_0 \in \R^n, \quad \mbox{constant }A \in Cl_n\,.
\end{align}

The spectral (CFT) representation of the Clifford Gabor wavelets \eqref{eq:CGW} is
\begin{align}
  &\mathcal{F}\{\psi^{c}\}(\bomega) 
  = \widehat{\psi^{c}}(\bomega) 
  \\ \nonumber
  &=
  A 
  (\underbrace{e^{-\frac{1}{2}\sum_k \sigma_k^2 (\omega_{k}-\omega_{0,k})^2}
  - e^{-\frac{1}{2}\sum_k \sigma_k^2 (\omega_{k}^2+\omega_{0,k}^2)}}_{=\phi(\bomega)\in \R})
  \,.
\end{align}
It follows with \eqref{eq:admconst} that the Clifford Gabor wavelet admissibility constant 
\begin{align}
  C_{\psi^{c}} 
  &= \int_{\mathbb{R}^n}
    \frac{\{\widehat{\psi^{c}}(\mbox{\boldmath $\xi$})\}^{\sim}\widehat{\psi^{c}}(\mbox{\boldmath $\xi$})}
    {|\mbox{\boldmath $\xi$}|^n}\,
    d^n\mbox{\boldmath $\xi$}
  \nonumber \\
  &= \widetilde{A}A \int_{\mathbb{R}^n}
    \frac{\phi(\mbox{\boldmath $\xi$})^2 }
    {|\mbox{\boldmath $\xi$}|^n}\,
    d^n\mbox{\boldmath $\xi$} \,.
\end{align}
If e.g. $A$ is a vector or a product of vectors (versor), then 
$C_{\psi^{c}}$ will be scalar.

\section{Conclusion}

We have introduced real Clifford (geometric) algebra wavelets for multivector signals taking values in $Cl_n$. \textit{Real} means that we completely avoid to use the field of complex numbers $\C$. This also applies to the use of a real Clifford (geometric) algebra Fourier transform for the spectral representation. An extension to $Cl_{0,n'}, n'=1,2(\mod 4)$ appears straight forward.

\section*{Acknowledgments}
  Soli deo gloria. I do thank my dear family, B. Mawardi, G. Scheuermann, D. Hildenbrand and V. Skala.

\bibliographystyle{plain}

\begin{thebibliography}{99}

\bibitem{PG:JPMorlet}
  P. Goupillaud, 
  \emph{Biography of Jean P. Morlet},
  \url{http://www.mssu.edu/seg-vm/bio_jean_p__morlet.html}

\bibitem{MM:CMRA}
  M. Mitrea, 
  Clifford Wavelets, Singular Integrals and Hardy Spaces. 
  \emph{Lect. Notes in Math.} \textbf{1575}, Springer, New York, 1994.

\bibitem{Brackx:etal}
  Brackx et al., 
  \black{\textit{Ghent Clifford Analysis Wavelet Publications}} (2001-2007),       
  \url{http://cage.ugent.be/crg/cliffordpublicaties.PDF}

\bibitem{LT:QW}
  L. Traversoni, 
  {Quaternion Wavelet Problems}.
  \emph{Proc. of 8th Int. Symp. on Approx. Theory}, 
  Texas A\&{}M University, Jan. 1995.
  Image Analysis Using Quaternion Wavelets.
  in E. Bayro-Corrochano, G. Sobczyk (eds.),
  \emph{Proc. of AGACSE 1999}, Birkh\"{a}user, Basel, 2001.

\bibitem{EBC:pubs}
  E. Bayro-Corrochano, \emph{List of Publications},
  \url{http://www.gdl.cinvestav.mx/edb/publications.html}

\bibitem{JL:QaAW} 
  J. Zhao, and L. Peng.
  Quaternion-valued
  Admissible Wavelets Associated with the 2D
  Euclidean Group with Dilations.
  \emph{J. of Nat. Geom.}, 2001;
  J. Zhao.
  Clifford algebra-valued
  Admissible Wavelets Associated with Admissible Group. 
  \emph{Acta Sci. Nat. Univ. Pek.}, \textbf{41}(5), (2005).

\bibitem{KCF:MonW}
  K\"{a}hler et al.
  Monogenic Wavelets over Unit Ball. 
  \textit{ZAA}, \textbf{24}, 813--824 (2005) .

\bibitem{SB:Wavelets}
  S. Bernstein, 
  T.E. Simos et al. (eds.).
  Clifford Continuous Wavelet Transforms in $L_{0,2}$ and $L_{0,3}$. 
  \emph{AIP Proc. of ICNAAM 2008}, \textbf{1048}, 634--637 (2008);
  S. Bernstein and S. Ebert,
  Spherical Wavelets, Kernels and Symmetries.
  \emph{AIP Proc. of ICNAAM 2009}, \textbf{1168}, 761--764 (2009);
  S. Bernstein, S. Ebert and R.S. Krausshar,
  Diffusion Wavelets on Conformally Flat Cylinders and Tori.
  \emph{AIP Proc. of ICNAAM 2009}, \textbf{1168}, 773--776 (2009);
  S. Bernstein,
  Spherical Singular Integrals, Monogenic Kernels and Wavelets
  on the 3D Sphere. 
  \textit{AACA}, \textbf{19}(2), 173--189 (2009).

\bibitem{HA:Groot-1}
   E. Hitzer and R. Ab\l amowicz, 
   Geometric Roots of $-1$ in Clifford Algebras $Cl(p,q)$ with $p+q \leq 4$. 
   Submitted to \emph{Adv. App. Cliff. Alg.}, May 2009. 
   Preprint version: Technical Report 2009-3, Department of Mathematics, 
   Tennessee Technological University, Tennessee, USA, May 2009. 
   Also available as: arXiv:0905.3019v1 [math.RA]

\bibitem{HM:CFT_UP_Cl3}
    B. Mawardi and E. Hitzer, 
    Clifford Fourier Transformation and Uncertainty Principle 
    for the Clifford Geometric Algebra $Cl(3,0)$. 
    \emph{Adv. App. Cliff. Alg.}, \textbf{16}(1), 41--61 (2006).

\bibitem{HM:ICCA7}
    E. Hitzer and B. Mawardi,
    Clifford Fourier Transform on Multivector Fields 
    and Uncertainty Principles for Dimensions 
    $n=2 \, (\rm mod \, 4)$ and $n=3 \, (\rm mod \, 4)$.
    \emph{Adv. App. Cliff. Alg.} 
    Vol. \textbf{18}(S3,4), 715--736 (2008).

\bibitem{HMA:2DCWFT}
    B. Mawardi, E. Hitzer and S. Adji, 
    Two-Dimensional Clifford Windowed Fourier Transform. 
    acc. for G. Scheuermann, E. Bayro-Corrochano (eds.), 
    \emph{Appl. Geom. Algs. in Comp. Sc. and Engineering}, 
    Springer, New York, 2010.


\bibitem{TAE:QFT}
   T. A. Ell, 
   Quaterion-Fourier Transform for Analysis of Two-dimensional 
   Linear Time-Invariant Partial Differential Systems. 
   \emph{Proc. 32nd IEEE Conf. on Decision and Control}, 
   Dec. 1993, 1830--1841.

\bibitem{EH:QFT}
    E. Hitzer, 
    Quaternion Fourier Transform on Quaternion Fields and Generalizations.      
    \emph{Adv. App. Cliff. Alg.}, \textbf{17}(3), 497--517 (2007);
    E. Hitzer, 
    Directional Uncertainty Principle for Quaternion Fourier Transforms.
    \emph{Adv. App. Cliff. Alg.}, Online First, 14 pp. (2009).

\bibitem{HMAV:WFTquat}
    B. Mawardi, E. Hitzer, R. Ashino and R. Vaillancourt, 
    Windowed Fourier transform of two-dimensional quaternionic signals. 
    Submitted to \emph{Appl. Math. and Comp.}, March 2009.
    
\bibitem{MH:CliffWUP}
    B. Mawardi and E. Hitzer,
    Clifford Algebra $Cl_{3,0}$-valued Wavelet Transformation, 
    Clifford Wavelet Uncertainty Inequality and Clifford Gabor Wavelets.
    \emph{Int. J. of Wavelets, Multires. and Inf. Proces.}, 
    \textbf{5}(6), 997--1019 (2007).

\bibitem{ES:CFTVectF}
    J. Ebling and G. Scheuermann, 
    Clifford Fourier Transform on Vector Fields. 
    \emph{IEEE Trans. on Vis. and Comp. Graph.}, 
    \textbf{11}(4), (July/Aug. 2005);
    Clifford Convolution And Pattern Matching On Vector Fields. 
    \emph{Proc. 14th IEEE Vis. 2003 (VIS'03)}, p. 26, (Oct. 22-24, 2003). 
    
\bibitem{EH:RCliffWT}
    E. Hitzer, T.E. Simos et al. (eds.).
    Real Clifford Algebra $Cl_{n,0}, n = 2,3(\text{mod} \,\,4)$ 
    Wavelet Transform.
    \emph{AIP Proc. of ICNAAM 2009}, \textbf{1168}, 781--784 (2009).

\bibitem{HM:UP2005}
    E. Hitzer and B. Mawardi, 
    Uncertainty Principle for the Clifford Geometric Algebra $Cl(3,0)$ 
    based on Clifford Fourier Transform. 
    in TE. Simos, G. Sihoyios, C. Tsitouras (eds.), 
    \emph{Proc. of ICNAAM 2005}, 
    Wiley-VCH, Weinheim, 922--925 (2005).

\bibitem{HM:UP2006}
    E. Hitzer and B. Mawardi, 
    Uncertainty Principle for Clifford Geometric Algebras 
    $Cl_{n,0} , n = 3 (\text{mod}\,\, 4)$ 
    based on Clifford Fourier Transform.
    in T. Qian, M.I. Vai, X. Yusheng (eds.), 
    \emph{Wavelet Analysis and Applications, 
    Springer (SCI) Book Series Applied and Numerical Harmonic Analysis}, 
    Springer, 45--54 (2006).

\bibitem{HM:UP2008}
    B. Mawardi, E. Hitzer, A. Hayashi and R. Ashino, 
    An Uncertainty Principle for Quaternion Fourier Transform, 
    \textit{Comp. \& Math. with Appl.}, \textbf{56}, 2398--2410 (2008).

 
    
\end{thebibliography}

\end{document}